\documentclass[runningheads]{llncs}

\usepackage{graphicx}
\usepackage{caption}
\usepackage{color}
\usepackage{amsfonts}
\usepackage{amsmath}
\usepackage{physics}
\usepackage{mathtools}
\usepackage[thinc]{esdiff}
\newcommand{\R}{\mathbb{R}}

\begin{document}

\title{Time-Optimal Control Problem With State Constraints In A Time-Periodic Flow Field
\thanks{This work was supported by FCT (Portugal): support of FCT R\&D Unit SYSTEC --
POCI-01-0145-FEDER-006933/SYSTEC funded by ERDF $|$ COMPETE2020 $|$ FCT/MEC $|$ PT2020
extension to 2018, and NORTE-01-0145-FEDER-000033, by ERDF $|$ NORTE 2020.
Results described in Chapter~\ref{sec: numercial results} were obtained
during research visit of R. Chertovskih to the Federal Research Center
``Computer Science and Control'' (Russian Academy of Sciences)
supported by the Russian Science Foundation (project no. 19-11-00258).
}
}

\titlerunning{Time-Optimal Control Problem In A Time-Periodic Flow Field}

\author{Roman Chertovskih\inst{1}\orcidID{0000-0002-5179-4344} \and
Nathalie T. Khalil\inst{1}\orcidID{0000-0002-4761-5868} \and
Fernando L. Pereira\inst{1}\orcidID{0000-0002-9602-2452}}
\authorrunning{Chertovskih, Khalil, Pereira}
%
\institute{Research Center for Systems and Technologies (SYSTEC), Electrical and Computer
	Engineering Department, Faculty of Engineering, University of Porto, 4200-465, Portugal
\email{roman@fe.up.pt, khalil.t.nathalie@gmail.com, flp@fe.up.pt}\\
\url{https://www.nathaliekhalil.com/, https://paginas.fe.up.pt/~flp/}}

\maketitle              

\begin{abstract}
This article contributes to a framework for a computational indirect method based on the Pontryagin maximum principle to efficiently solve a class of state constrained time-optimal control problems in the presence of a time-dependent flow field.
Path-planning in a flow with tidal variations is an emblematic real-life task that serves as an example falling
in the class of problems considered here. Under rather natural assumptions on the flow field, the problem is
regular and the measure Lagrange multiplier, associated with the state constraint, is continuous.
These properties (regularity and continuity) play a critical role in computing the field of extremals by solving
the two-point boundary value problem given by the maximum principle. Moreover, some sufficient conditions
preventing the emergence of singular control processes are given.
A couple of examples of time-periodic fluid flows are studied and the corresponding optimal solutions are found.

\end{abstract}

\section{Introduction}

In this article, we consider a state-constrained time-optimal control problem in the presence of a time-periodic flow field, the so-called ``navigation problem''. We are interested in computing its set of extremals using an indirect method based on the necessary optimality conditions provided by the maximum principle
\cite{Pontryagin_1962}. However, indirect methods represent significant challenges for optimal control problems with
state constraints. Indeed, the difficulty is due to the fact that the state constraint Lagrange multiplier appearing
in the necessary optimality conditions whenever the state constraint becomes active is a mere Borel measure. Thus,
in general, this multiplier is discontinuous, and this leads to serious difficulties in computing the set of extremals at the times in which the state trajectory meets the boundary of the state constraint set.
Therefore, in order to overcome this difficulty, we employ a not so commonly used form of the maximum principle - the
so-called Gamkrelidze form - and impose a regularity condition on the data of the problem that entails the continuity of the measure Lagrange multiplier. This is crucial to ensure the appropriate behavior of the proposed numerical
procedure to find the corresponding set of extremals.

Moreover, we also present certain conditions that, once satisfied, prevent the emergence of singular control
processes. These may be helpful in guiding the computational procedures, by enabling to check the absence of singular controls.

To better grasp the proposed indirect method in the framework of regular problems, we study a navigation problem
in $\mathbb R^2$. More precisely, we consider an object moving in a closed state domain, subject to a time-dependent
fluid flow vector field. The dynamics of the proposed model is affine in the control variable whose values are
constrained to the unit square in $\R^2$, and is affected by the vector flow field action. For this model,
we are interested in computing the optimal time of a path which connects two given distinct initial and terminal
points. For the problem in question, the regularity condition is satisfied under non-restrictive conditions on the
vector flow field. In order to develop the proposed indirect method, we derive the corresponding necessary
optimality conditions in the non-degenerate Gamkrelidze form, from which the expression of the optimal control is
computed.
Moreover, the regularity of the problem entails an explicit expression for the measure multiplier. The points where
the extremal trajectories reach the boundary of the state constraint can be computed as a result of the continuity of
the measure Lagrange multiplier. The two expressions - of the control and the measure multiplier - are functions of
the state and adjoint variables, and are replaced in the associated two-point boundary value problem, solved via a
shooting algorithm. From the set of all extremals, only the optimal ones, i.e., with the minimal time, are selected as
solutions to the given time-optimal problem. We discuss some examples of time-periodic vector flow fields, and
we plot the corresponding set of extremals.

The proposed numerical approach based indirect method was discussed in \cite{Khalil et al. IEEE TAC 2019} for the steady flow field case. Our paper extends the analysis to time-periodic flow fields. The dependence of the flow on time is crucial for many realistic path planning problems. For example, tides play an important role in shaping water
velocity fields in rivers, mainly near their mouth~\cite{jos}. Moreover, in our paper we discuss, with more details, sufficient conditions for the non-occurrence of singular controls for a particular choice of control set.

The area of state-constrained optimal control problems has been widely investigated in the literature, cf. \cite{Gamkrelidze_1960,Berkovitz_1962,Warga_1964,Gamkrelidze_1965,Dubovitskii_Milyutin_1965,Hestenes_1966,Neustadt_1966,Neustadt_1967,Halkin_1970,Russak_1970,Ioffe_Tikhomirov_1979}. Questions related to the non-degeneracy of the necessary optimality conditions can be found in \cite{Arutyunov_Tynyanskiy_1985,Dubovitskii_Dubovitskii_1985,Vinter_Fereira_1994,Arutyunov_Aseev_1997,Arutyunov_2000,Vinter_2000,Arutyunov_Karamzin_Pereira_2011,Bettiol_Khalil_2016}. Issues on the continuity of the measure multiplier are extensively studied in \cite{Hager_1979,Maurer_1979,Milyutin_1990,Galbraith_Vinter_2003,Bonnans_2009,Arutyunov_2012,Arutyunov_Karamzin_2015,Arutyunov_Karamzin_Pereira_2018}. Numerical methods for computing the set of extremals in the presence of state constraints can be found in \cite{Bryson_1969,Jacobson_1969,Betts_1993,Fabien_1996,Maurer_2000,Vasiliev_2002,Pytlak_2006,Haberkorn_2011,Keulen_2014}. For indirect methods, we refer the reader to \cite{Jacobson_1969,Fabien_1996,Maurer_2000} \cite{Vasiliev_2002,Pytlak_2006,Haberkorn_2011,Keulen_2014}, among many others.

The article is organized as follows: the formulation of the time-optimal navigation problem is described in Section \ref{sec: problem formulation} and the regularity concept is discussed. In Section \ref{sec: maximum pronciple}, the necessary optimality conditions in the Gamkrelidze's form are given in a non-degenerate form. Section
\ref{sec: applications} is devoted to the application of the maximum principle to the problem in question when the
control set is constrained to the unit square in $\R^2$ and to derive the explicit formulas for the corresponding
measure multiplier and extremal control. Sufficient conditions for the non-occurrence of singular controls are also discussed.
Numerical results and the description of the algorithm are featured in Section \ref{sec: numercial results}. Section \ref{sec: conclusion} concerns a conclusion, and the Appendix contains detailed proofs of the key results.

\section{Problem Formulation: Navigation Problem}\label{sec: problem formulation}

We consider an object driven by a dynamical system in a two-dimensional time-space dependent flow field $v(t,x)$, and while subject to affine state constraints. The ultimate goal is to compute a control process that yields the minimum  transit time between two given starting and final points $A$ and $B$ within the set of extremals, i.e., the set of control processes satisfying the maximum principle conditions. The corresponding problem is described as follow:

\begin{equation}
\begin{aligned}\label{problem}
& {\text{Minimize}}& & T  \\
& \text{subject to} & &  \dot x = u + v(t,x), \\
&&& x(0)=A,\;\;x(T)=B, \\
&&&  |x_1|\le 1, \\
&&& u \in U: = \{ u \ : \ \varphi(u) \leq 0   \},
\end{aligned}
\end{equation}
where $x={\it col}(x_1,x_2)\in AC ([0,T]; \R^2)$, and $u={\it col}(u_1,u_2)\in L^1([0,T]; \R^2)$ are, respectively, the state and the control variables. The point $A$ is the starting point, while $B$ is the terminal point, and $v:[0,T] \times \mathbb{R}^2 \to \mathbb{R}^2$ is a smooth map which defines a fluid flow varying in time and space. Moreover, $\varphi=(\varphi_1, \varphi_2)$ is a given vector-valued mapping  representing the control constraint set $U\subset \R^2$. The state constraint set is represented by the inequality $|x_1| \le 1$.
The terminal time $T$ is to be minimized by the optimal control process.

\subsection{Regularity Condition}

For the problem (\ref{problem}) we consider here, the function $\Gamma(t,x,u): [0,T] \times \R^2 \times \R^2 \to \R$, defined by the scalar product of the gradient of the function defining the considered active inequality state constraint and the corresponding dynamics, as

\[ \Gamma(t,x,u) := u_1 + v_1(t,x)  . \]

Denote by $I_\varphi(u)$ the set of $i$'s such that $\varphi_i(u)=0$ ($i=1,2$). In the sequel of the regularity concept in \cite{Arutyunov_Karamzin_2015}, we state the following definition:
\begin{definition}\label{def: regularity}
Assume that, for all $t\in [0,T]$, $x\in\mathbb R^2$ and $u\in \R^2$, such that $|x_1|=1$, $\varphi(u)=0$ and $\Gamma(t,x,u)=0$, the set of vectors $\pdv{\Gamma}{u}$ and $\nabla \varphi_i$, for all $i\in I_\varphi(u)$ is linearly independent. Then, we say that problem (\ref{problem}) is regular with respect to the state constraint.
\end{definition}

As it will be explained in the coming sections, the regularity of the problem is, in the context of our paper, crucial for the appropriate behavior of the numerical proposed approach at points in which the trajectory meets the boundary of the state constraint set. The regularity condition might seem restrictive. However, for a large class of engineering problems it is automatically satisfied under natural assumptions. An example will be featured in Section \ref{sec: applications} for a specific case of control set $U$.

\section{Maximum Principle}\label{sec: maximum pronciple}

In this section, we derive non-degenerate necessary optimality conditions in the Gamkrelidze's form
for problem (\ref{problem}).
We start by considering the extended time-dependent Hamilton-Pontryagin function
$$
\bar H(t,x,u,\psi,\mu,\lambda) = \langle\psi,u + v(t,x)\rangle - \mu \Gamma(t,x,u) -\lambda,
$$
where $\psi \in \mathbb R^2$, $\mu \in \mathbb R$ and $\lambda \in \mathbb R$.

In order to satisfy the notation in what follows, we denote by $f^*(y,z)$ the function $ f(x,y,z)$ in which $x$ is
replaced by the reference value $x^*$.

\begin{theorem}\label{theorem: maximum principle}
We assume that problem (\ref{problem}) is regular in the sense of Definition \ref{def: regularity}. Then, for an optimal process $(x^*,u^*,T^*)$, there exist a set of Lagrange multipliers: a number $\lambda \in [0, 1]$, an absolutely continuous adjoint arc $\psi=(\psi_1,\psi_2) \in W_{1,\infty}([0, T^*]; \mathbb R^2)$, and a scalar function $\mu(\cdot)$, such that:
\begin{itemize}
\item[(a)]\label{item: adjoint system}Adjoint equation
    \begin{align*} \dot \psi(t) & = -\frac{\partial{\bar H}}{\partial x}(t,x^*(t),u^*(t),\psi(t),\mu(t),\lambda)  
    \qquad {\textrm for \ a.a. }\ t\in [0,T^*];
    \end{align*}
\item[(b)]\label{item: max condition} Maximum condition
    \begin{align*}u^*(t) & \in \mathop{\rm argmax}_{\varphi(u) \le 0} \left\{{\bar H} (t,x^*(t),u,\psi(t),\mu(t),\lambda) \right\} \qquad {\textrm for \ a.a. }\ t\in [0,T^*];
    \end{align*}
\item[(c)]\label{item: conservation law} Time-transversality condition
    \[ h(T^*)=0 \; \mbox{ where }\; h(t):= 	\max_{\varphi(u) \le 0} \left\{ \bar H^*(t,u) \right\} \; \mbox{ satisfies } \; \dot h= \pdv{\bar H^*}{t} (t)\]
    a.a. $ t \in [0,T^*]$;
\item[(d)]\label{item: measure continuity} $\mu(t)$ is constant on the time intervals where $ |x^*(t)|< 1 $,
	increasing on  $\{t\in [0,T]:x_1^*(t)=-1\}$, and decreasing on $\{t\in [0,T]: x_1^*(t)=1\}$.
	Moreover, $\mu(\cdot)$ is continuous on $[0,T^*]$;
\item[(e)] \label{item: nontriviality condition} Non-triviality condition
    $$\begin{array}{ll} &\lambda + |\psi_1(t)-\mu(t)| + |\psi_2(t)|  > 0,\end{array} \qquad  \text{for all }\ t\in [0,T^*].$$
\end{itemize}
\end{theorem}

\begin{remark}\label{remark: extended problem}

\noindent
\begin{itemize}
\item[a)] The proof of Theorem \ref{theorem: maximum principle} can be found in the Appendix. It relies on a time-reparametrization technique converting problem (\ref{problem}) into a fixed and time-independent one, and on results of \cite{Arutyunov_2000}, \cite{Karamzin_Pereira_2019}, and \cite[Theorem 4.1]{Arutyunov_Karamzin_Pereira_2011}.
\item[b)] From the regularity property of the problem, the expression of the measure multiplier can be found in terms of the state and the adjoint variables. The junction points -- the points at which the trajectory meets the state constraint boundary -- can be computed as a result of the continuity of the measure multiplier $\mu$ (condition (d)). From these considerations, explicit formulae for the measure multiplier can be obtained. This is the core of our computational scheme proposed in this article for finding the set of extremals.
\item[c)] The non-triviality condition (e), which asserts the non-degeneracy of the Maximum Principle, implies that
    \begin{equation} \label{eq: result from the non-triviality} |\psi_1(t)-\mu(t)| + |\psi_2(t)|  > 0 \quad \text{for all } t \in [0,T^*]. \end{equation}
	A proof is provided in the Appendix.
\end{itemize}

\end{remark}

\section{Applications: Control Set Constrained to the Square} \label{sec: applications}

In this section, we consider the specific case of a control set represented as the unit square in $\R^2$, i.e.

    \begin{equation} \label{def: control set square}  U: = \{ u \in \R^2 \ : \ \varphi_1(u):= |u_1| \le 1 \text{ and }  \varphi_2(u) := |u_2| \le 1 \}.
    \end{equation}

We study how a simple assumption on the vector flow field can automatically lead to regularity in the sense of Definition \ref{def: regularity}. Thereafter, we use the necessary optimality condition derived in Section \ref{sec: maximum pronciple} to obtain explicit expressions of the optimal control and the measure Lagrange multiplier in terms of the state and adjoint variables. These expressions will be substituted in the associated
boundary-value problem to numerically find the set of extremals (see Section \ref{sec: numercial results}).

\subsection{Sufficient Condition for Regularity}

In the problem considered here, the following simple assumption suffices to ensure regularity as defined above.

\begin{itemize}
\item[(H)] $|v_1(t,x)| < 1$ for all $(t, x)\in \R\times\R^2$.
\end{itemize}

Indeed, if the starting and/or terminal positions are in the interior of the state constraint domain, and if
the flow is much faster at the boundary of the state constraint set, $|x_1|=1$, assumption (H) is crucial to
guarantee that the moving object is able to overcome the flow field effect, and, thus, leave the boundary of
the state constraint, and move across the river along the axis $0x_1$.

\begin{proposition}
Assume that (H) is satisfied. Then, the problem (\ref{problem}), with the specific choice of the control set $U$, as defined in (\ref{def: control set square}), is regular in the sense of Definition \ref{def: regularity}.
\end{proposition}

\begin{proof}
For $\Gamma(t,x,u)=0$, $|u_1| < 1$ ${\cal L}$-a.e., and the strict inequality $\varphi_1(u)< 0$ holds at the boundary of the state constraint set. Then, the regularity condition is satisfied if the vectors $\nabla\varphi_2$, and $\pdv{\Gamma}{u}$ constitute a linearly independent set. For our particular problem, $\nabla\varphi_2= {\it col}(0,1)$, and $\pdv{\Gamma}{u} = {\it col}(1,0)$. Therefore, the problem (\ref{problem}) with the particular choice in (\ref{def: control set square}) is regular.
\end{proof}

\begin{remark}
Under assumption (H), the necessary conditions of optimality expressed by Theorem \ref{theorem: maximum principle}, guaranteeing the non-degeneracy of the Lagrange multipliers, and the continuity of the Borel measure $\mu$, can be applied.
\end{remark}

\subsection{Explicit formulas for $u^*$ and $\mu$}

From the maximum condition (b) of Theorem \ref{theorem: maximum principle}, for a.a. $t\in [0,T^*]$,
    \begin{align*} & \max_{|u_1| \le 1,|u_2| \le 1} \left\{\left(\psi_1(t) - \mu(t)\right)u_1 + \psi_2(t) u_2\right\}  = \left(\psi_1(t) - \mu(t)\right)u_1^*(t)+ \psi_2(t) u_2^*(t).
    \end{align*}
This implies that the value of the optimal control process $(u_1^*,u_2^*)$ varies w.r.t. the sign of $\psi_1-\mu$ and $\psi_2$, as follows:
\begin{equation}
\label{uop2}
\begin{cases}
\text{if}\ \psi_1 - \mu \neq 0, & \text{then}\ u_1^* = sgn(\psi_1 - \mu)\\
\text{if}\ \psi_2 \neq 0, & \text{then}\ u_2^* = sgn(\psi_2).\\
\end{cases}
\end{equation}
The expressions of $u^*$ and $\mu$ differ for points belonging to the boundary of the state constraint set or for points in its interior. Next, we discuss these two cases.

When the trajectory stays on the boundary of the state constraint set during a certain set $\Delta$, then $u_1^*(t)=-v_1(t,x^*(t))$ ${\cal L}$-a.e. on $\Delta$, and, by continuity, everywhere on $\Delta$. Thus, under assumption (H), $|u^*_1(t)|<1$. Therefore, from the maximum condition, we have
\begin{equation}
\label{mu2}
\mu(t)=\psi_1(t) \qquad \text{for all } t \text{ such that } |x^*(t)|=1.
\end{equation}

Moreover, as a result of (\ref{eq: result from the non-triviality}), we have $\psi_2(t) \neq 0$ for all $t$ such that $|x_1^*(t)|=1$ and, thus, $u_2= \pm 1$.

Now, let $\Delta$ be a time interval during which the trajectory lies in the interior of the state constraint set, i.e. $|x_1^*(t)|<1$ $\forall\, t\in\Delta$. For any point $t\in \Delta$, the Lagrange measure multiplier $\mu$ is constant. Thus, it follows that, for some $t$ in a nonzero Lebesgue measure subset of $ \Delta$, we have $\psi_1(t) - \mu(t) = 0$, and since $\mu(t) $ is constant on $\Delta$, and hence, $\dot \psi_1 (t)=0 $, on this subset. From the adjoint equations, we conclude that $\displaystyle \psi_2(t)\frac{\partial v _2(x)}{\partial x_1} = 0$. From the non-triviality condition, we have to have $\psi_2(t) \neq 0 $ ${\cal L}$-a.e. on $\Delta$, and, thus, $\displaystyle \frac{\partial v_2(x)}{\partial x_1} = 0$ ${\cal L}$-a.e. on $\Delta$. In this case the maximum condition is not informative for the first component of the control. This is one case of the so-called singular control, i.e., the controls cannot be defined on a non-zero measure set. Another singular situation corresponds to the case when $\psi_2(t)=0$ for some $t$ in a nonzero Lebesgue measure subset of $\Delta$. Next, we present and prove sufficient conditions that preclude the emergence of singular control by imposing additional conditions on the flow vector-field $v$.

Let $ S_0=\{t\in[0,T]: |x_1(t)|=1\}$, and $S_{-}=\{t\in[0,T]: |x_1(t)|<1\}$, $\Delta\subset [0,T]$ such that ${\cal L}\mbox{-meas}(\Delta) >0 $, and
\begin{eqnarray*}
S_1&=&\{t\in\Delta\subset S_{-}: \psi_1(t)-\mu(t) =0 \mbox{ and }\psi_2\neq 0 \}\\
S_2&=&\{t\in\Delta\subset S_{-}: \psi_1(t)-\mu(t) \neq 0 \mbox{ and }\psi_2 = 0 \}\\
S_3&=&\{t\in\Delta\subset S_0: \psi_1(t)-\mu(t) =0 \mbox{ or }\psi_2 = 0 \}\\
\end{eqnarray*}
In what follows, we suppress the $t$-dependence of $v$ as it does not play any role in the developments.

\begin{proposition}\label{proposition: regularity sufficient condition-2}

\noindent
\begin{itemize}
\item[a)] If $\displaystyle \frac{\partial v_2(x(t))}{\partial  x_1} \neq 0$ on $S_1$, then ${\cal L}\mbox{-meas}(S_1)=0 $.
\item[b)] If $\displaystyle \frac{\partial v_1(x(t))}{\partial  x_2} \neq 0$ on $S_2$, then ${\cal L}\mbox{-meas}(S_2)=0 $.
\item[c)] On $S_3$ we always have  ${\cal L}\mbox{-meas}(S_3)=0 $.
\end{itemize}
\end{proposition}

\begin{proof}

\noindent
Proof of item $a)$. Clearly, for all $t\in S_1$, $\displaystyle - \dot \psi_1(t) = \psi_2(t)\frac{\partial v_2(x(t))}{\partial x_1}$, and, thus $-\dot \psi_1(t) \neq 0$ for all $t\in S_1$. Since $S_1\subset S_{-}$,
we readily conclude that ${\cal L}$-meas$(S_1)=0$.

\noindent Proof of item $b)$.  Now, for $t\in S_2$, we have $\displaystyle - \dot \psi_2(t) = (\psi_1(t)-\mu(t)) \frac{\partial v_1(x(t))}{\partial x_2}$. Thus, if $\frac{\partial v_1(x(t))}{\partial x_2}\neq 0$, then $\dot \psi_2 (t) \neq 0$, and, thus, ${\cal L}$-meas$(S_2)=0$.

\noindent Proof of item $c)$. Consider some $t\in S_3$. The system of adjoint equations can be written as follows
\begin{equation}\label{adjoint2}
-\dot \psi(t) = D_x^T v(x(t))\psi(t)-\mu(t) \nabla_x v_1(x(t))
\end{equation}
Since $\psi_1(t) -\mu(t) = 0 $ on $ S_0$, we have
$$ -\dot \psi_1(t) = \frac{\partial v_2(x(t))}{\partial x_1} \psi_2(t) \; \mbox{  and } -\dot \psi_2(t) = \frac{\partial v_2(x(t))}{\partial x_2} \psi_2(t).$$

If there exists some $\bar t \in \Delta$ such that $\psi_2 (\bar t) =0$, then  $\psi_2 (\bar t) =0$ on a nonzero Lebesgue measure subset of $\Delta$. From the above, also follows that $\psi_1(t) $ is constant on this subset what contradicts $\psi_1(t)-\mu(t) =0 $ on $S_0$. From this, and the nontriviality condition of the multipliers, we have the desired conclusion.


\end{proof}



	
	
\begin{remark}\label{remark: extra regularity condition sufficient condition}

Proposition \ref{proposition: regularity sufficient condition-2} represents sufficient conditions to avoid singular controls. These conditions allow
the maximum condition to stay informative and to define the optimal controls on a non-zero measure set.
Since the conditions are only sufficient, this means that there might be some cases in which the singular controls do not occur even if these conditions are not satisfied.

\end{remark}

	
%

\section{Numerical Results}\label{sec: numercial results}

In this section, we present and discuss some numerical results for the above stated optimal control problem by 
using an indirect method based on the considered Maximum Principle. The conditions of this Maximum Principle 
lead to the following two-point Boundary Value Problem (BVP):
\begin{subequations}
\label{goveq}
\begin{align}
\dot{x}&=u+v,\label{ge1}\\
\dot{\psi}&=-\psi \frac{\partial v}{\partial x} + \mu \nabla v_1,\label{ge2}\\
 &x(0)=A,\label{xini}\\
 &x(T)=B\label{xterm}.
\end{align}
\end{subequations}
The control variables $u_1(t)$ and $u_2(t)$ are given by (\ref{uop2}) and the measure Lagrange multiplier 
$\mu(t)$ is constant for trajectories not meeting the boundary and is defined by (\ref{mu2}) along the boundary 
of the state constraint.
Boundary conditions at $0$ and $T$ for the adjoint arc $\psi(t)$ are absent.

The BVP problem (\ref{goveq}) is solved by a variant of the shooting method (see, e.g., \cite{nr} for a brief overview of the shooting methods). The shooting parameter is the angle $\theta$ parameterizing the initial boundary condition for $\psi$:
\begin{equation}
\label{psiini}
\psi(0)=(\cos(\theta),\sin(\theta)).
\end{equation}
Starting from the initial conditions (\ref{xini}) and (\ref{psiini}) for a given value of $\theta$, the Cauchy 
problem for the system of ordinary differential equations (\ref{ge1})--(\ref{ge2}) is solved by the classical
$4^{th}$ order Runge-Kutta method with the constant time step $\tau=10^{-4}$.
The set of solutions to the BVP (\ref{goveq}) constitute the field of extremals. By integrating the system (\ref{ge1})--(\ref{ge2}) forward in time, the measure Lagrange multiplier $\mu(t)$ is 
set to zero for the trajectory in the interior of the domain (i.e. when $|x_1|<1$).
If it reaches the boundary, $|x_1|=1$, at, say, $t=t^*$, then $\mu^*=\mu(t^*)$ is computed by (\ref{mu2}).
By using the continuity of the measure Lagrange multiplier \cite{Arutyunov_Karamzin_2015}, we deem, that if $|\mu^*|<10^{-3}$, the point is a junction point of an extremal and integration of the system is continued 
following the domain boundary. At each time step along the boundary, the trajectories leaving it with constant  values of $\mu$ are computed.
This is done to find another junction point or a segment joining the boundary with the terminal point $B$.
If at a certain time the terminal boundary condition (\ref{xterm}) is satisfied to the accuracy $10^{-3}$,
the corresponding trajectory represents an extremal.
To find all extremals, the parameter $\theta$ is varied from $0$ to $2\pi$ in (\ref{psiini}) with
a constant step of $10^{-2}$, being the bisection method used if the required accuracy is not achieved.
Once all extremals, i.e., field of extremals) are computed, the one possessing minimal travelling time among 
all the extremals is the optimal solution to the original control problem (\ref{problem}).



The first example considers the steady fluid flow
\begin{equation}
\label{eq:steady}
v(x)=(\frac{1}{4}x_1,\, -x_1^2),
\end{equation}
which is represented by the black arrows in Fig.~\ref{fig1}. We clearly notice that, for this specific field, $\displaystyle \pdv{v_1}{x_2}=0$ everywhere, while $\displaystyle \pdv{v_2}{x_1}=0$ at points such that $x_1=0$. However the field of extremal can still be numerically found, supporting Remark \ref{remark: extra regularity condition sufficient condition}. The initial and final positions are $A=(0,0)$ and $B=(-0.5,-6)$, respectively. The field of extremals is also shown in Fig.~\ref{fig1}. The optimal extremal is the red one with travelling time $3.43$ s time units.


In the next example, a perturbation of the flow periodic in time (\ref{eq:steady}) is considered for the same 
trajectory endpoints $A$ and $B$. Here,
$$v(x,t)=\left(\frac{1}{4 x_1}+\sin(\frac{\pi t}{2}),\, -x_1^2\right),$$
reflects the fact that tidal variations affect only the component transversal to the main flow. Even Proposition \ref{proposition: regularity sufficient condition-2} is violated at some points, the field of extremals is computed and it is shown in Fig.~\ref{fig2}, displaying four extremals as in the steady case considered above. In contrast to the problem for the steady flow (\ref{eq:steady}),
only one extremal (shown in blue) does not meet the boundary; two extremals (black) meet the right boundary and 
only one (red) has the left boundary segment. The optimal extremal is again the one (shown in red, with 
travelling time $3.63$ s) involving the left boundary segment, however, the travelling time along the extremal
(black) involving the right boundary, $3.77$ s, is not significantly larger.

\begin{figure}[h!]
	\begin{minipage}{0.45\textwidth}
		\centering
		\includegraphics[width=\textwidth]{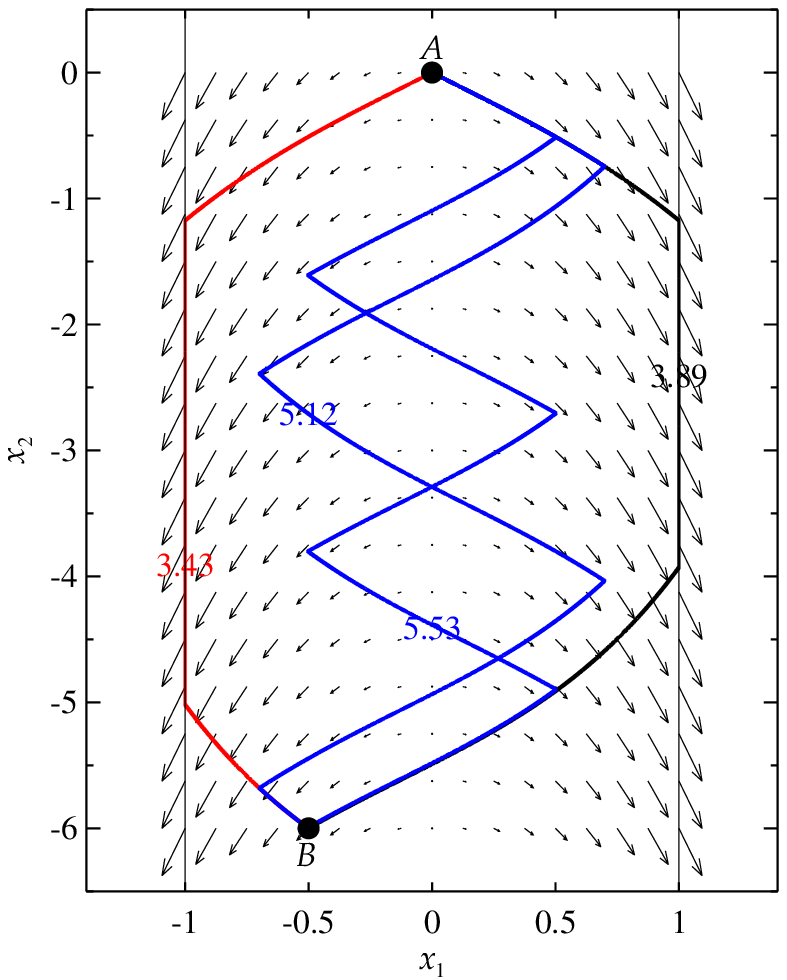}
		\caption{\small Field of extremals for the steady fluid flow \hbox{$v(x)=(\frac{1}{4}x_1,\, -x_1^2)$}.
			Extremals not meeting the boundary are shown in blue, meeting the right and left boundaries
			are displayed in black and red, respectively.
			Inscribed numbers stand for travelling times along the
			corresponding extremals.}
		\label{fig1}
	\end{minipage}\hfill
	\begin{minipage}{0.45\textwidth}
		\centering
		\includegraphics[width=\textwidth]{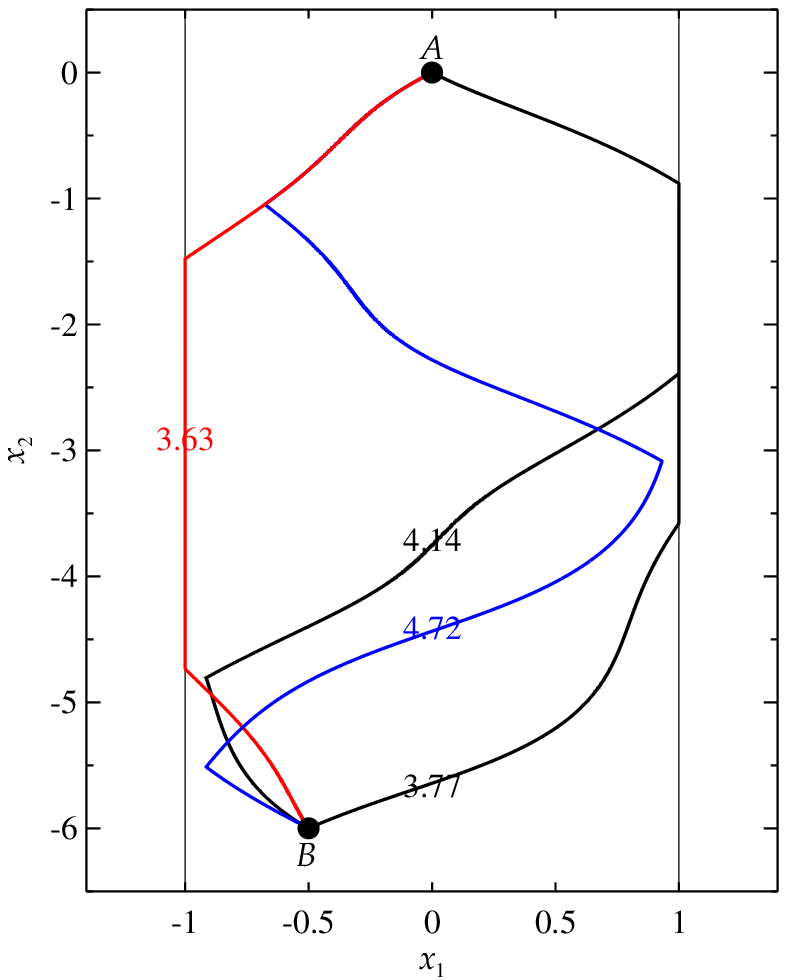}
		\caption{\small Field of extremals for the time-periodic fluid flow \hbox{$v(x,t)=(\frac{1}{4}x_1+\sin(\pi t/2),\, -x_1^2)$}. Extremals not meeting the boundary are shown in blue, and meeting the right and left boundaries
			are displayed in black and red, respectively.
			Inscribed numbers stand for travelling times along the corresponding extremals.}
		\label{fig2}
	\end{minipage}
\end{figure}

The last example concerns the flow vector field
\begin{equation*}
\label{eq:steady-time}
v(x)=\bigg(\frac{x_1}{4} + \frac{x_2}{10},\, -x_1^2- \frac{1}{2}\sin^2\bigg(\frac{\pi t}{2}\bigg)\bigg),
\end{equation*}

For this field, $\displaystyle \pdv{v_1}{x_2} \neq 0$, while $\displaystyle \pdv{v_2}{x_1} = 0$ when $x_1=0$. The corresponding field of extremals is computed and shown in Fig. \ref{fig3}. This field contains five inner trajectories not touching the state constraint boundary, one right trajectory, and an optimal one that reaches the final point $B$ after $3.19$ time units.

\begin{figure}
	\centerline{
		\includegraphics[width=0.45\textwidth]{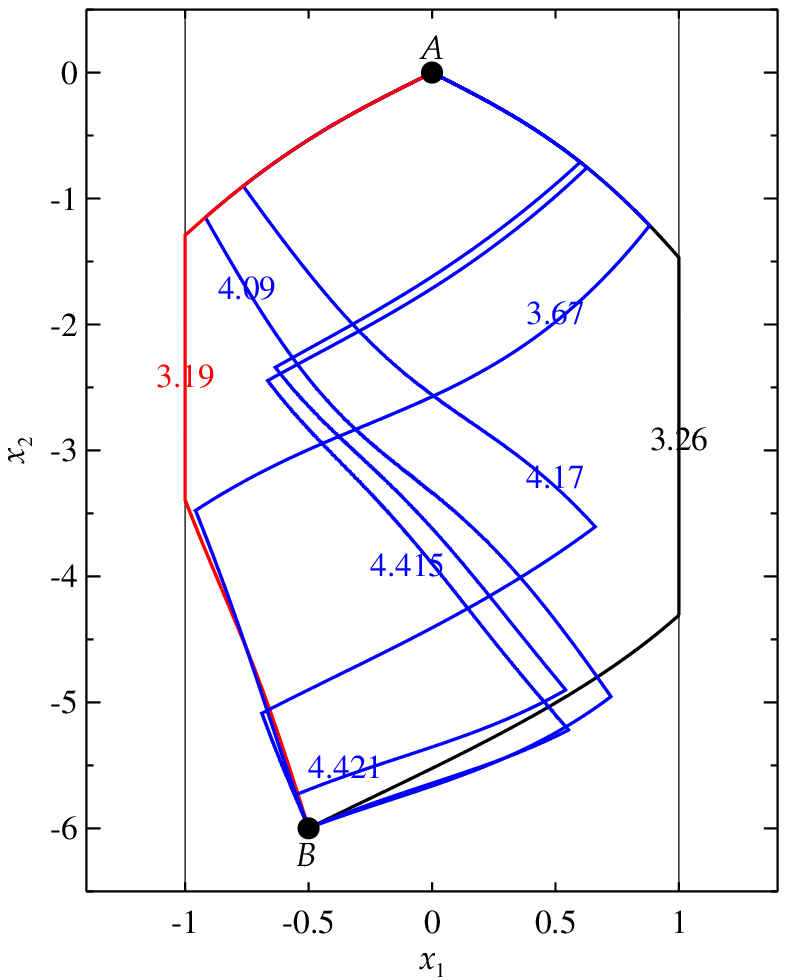}
	}
	\caption{
		Field of extremals for the fluid flow \hbox{$v(x)=\left(\frac{1}{4}x_1 + \frac{1}{10} x_2,\, -x_1^2- \frac{1}{2}\sin^2\left(\frac{1}{2}\pi \right)\right)$}.
		Extremals not meeting the boundary are shown in blue, meeting the right and left boundaries
		are displayed in black and red, respectively. Inscribed numbers stand for travelling times along the 
		corresponding extremals.
	} \label{fig3}
\end{figure}

\section{Conclusion}\label{sec: conclusion} In this article, we presented an approach based on the maximum principle amenable to the numerical 
computation of solutions to a regular class of state-constrained optimal control navigation problems 
subject to a flow field effects. In order to overcome the computational difficulties due to the Borel measure
associated with the state constraints, a not so common version of the maximum principle - the so-called Gamkrelidze form - was adopted and a regularity condition on the data of the problem was imposed to ensure 
the continuity of the Borel measure multiplier. 
We showed how this property plays a significant role for trajectories meeting the state constraint boundary. 
We also proved that this regularity assumption is not restrictive, and it is naturally satisfied by a wide 
class of optimal control problems. 
The theoretical analysis was supported by several illustrative examples (for steady and time-periodic
flows mimicking real river currents) for which the corresponding fields of extremals were constructed, and optimal solutions were found.

\section*{Appendix}
In the first part of this appendix, the proof of Theorem \ref{theorem: maximum principle} is provided. The second part concerns the proof of condition (\ref{eq: result from the non-triviality}).

{\bf Proof of Theorem \ref{theorem: maximum principle}.} Following Remark \ref{remark: extended problem} a), a time-reparametrization is needed. Indeed, for a minimizer $(x^*,u^*,T^*)$ of the original problem (\ref{problem}), we consider the extended fixed and independent-time optimal control problem associated to (\ref{problem}).


\begin{equation}
\begin{aligned}\label{problem:extended}
& {\text{Minimize}}
& & x_0(T^*) \\
& \text{subject to}
& &  \dot x_0 = u_0, \quad  \dot x = (u + v(x_0,x))u_0, \quad \text{ a.e. } t \in [0,T^*] \\
&&& x(0)=A,\;\;x(T^*)=B, \; \; x_0(0)=0,\;\;x_0(T^*) \in \R\\
&&&  |x_1|\le 1,  \quad  \text{ for all }\ t \in [0,T^*]\\
&&& u_0 \in [1-\alpha, 1+\alpha]   \quad  \text{ a.e. } t \in [0,T^*] \\
&&& u \in U: = \{ u \ : \ \varphi(u) \leq 0   \},
\end{aligned}
\end{equation}
where $x_0$ is a new state variable which parameterizes the time, and $u_0$ the associated control defined in $[1-\alpha,1+\alpha]$ where $\alpha\in (0,1)$. We can prove that if $(x^*,u^*,T^*)$ is a minimizer for (\ref{problem}) then $({x}_0^*(t)=t, u_0^*= 1, x^*, u^*)$ is a minimizer for (\ref{problem:extended}).

The corresponding time-independent Hamiltonian is denoted $\bar H_0$ and defined as follow:
\[ \bar H_0 := \bar H u_0 + \psi_0u_0  \qquad \text{where } \psi_0 \in \R. \]
Here $\bar H$ is the extended Hamiltonian for the original problem (\ref{problem}).
The application of the maximum principle to (\ref{problem:extended}) yields the existence of a number 
$\lambda\in [0,1]$, $(\psi_0,\psi) \in W_{1,\infty}([0,T^*];\R) \times W_{1,\infty}([0,T^*];\R^2)$, and a 
scalar function $\mu(.)$, such that
\begin{itemize}
\item[(i)] $
	(\dot \psi_0 (t), \dot \psi(t))  =  -\pdv{\bar H_0}{x_0} \times \pdv{\bar H_0}{x}
	({x}_0^*, u_0^*, x^*,u^*,\psi_0,\psi,\mu,\lambda) \qquad \text{ for a.e. } t \in [0,T^*];
	$
\item[(ii)] $(u_0^*(t),u^*(t))  \in \mathop{\rm argmax}\limits_{w\in [1-\alpha, 1+\alpha], \ \varphi(u) \le 0} \{{\bar H_0}(x_0^*(t), u_0,x^*(t),u,\psi_0(t), \psi(t),\mu(t),\lambda) \}$
	for a.e. $t\in [0,T^*]$;
\item[(iii)] (Conservation law)
	$ \max\limits_{u_0\in [1-\alpha, 1+\alpha],\ \varphi(u) \le 0} \{\bar H_0(u_0,u) \} =0  \qquad \forall t\in [0,T^*]; $
	\item[(iv)] $\mu(t)$ is constant on the time intervals where
	$ |x^*(t)|< 1,$
	increasing on  $\{t\in [0,T]:x_1^*(t)=-1\}$, and decreasing on $\{t\in [0,T]: x_1^*(t)=1\}$.
	Moreover, $\mu(\cdot)$ is continuous on $[0,T^*]$;
\item[(v)]
	$\lambda + |\psi_0(t)| + |\psi_1(t)-\mu(t)| + |\psi_2(t)| > 0, \qquad \forall t\in [0,T^*].$
\end{itemize}

\begin{remark} These necessary optimality conditions are results of \cite{Arutyunov_2000}, \cite{Karamzin_Pereira_2019}, and \cite[Theorem 4.1]{Arutyunov_Karamzin_Pereira_2011}. Moreover, the non-triviality condition (v) is implied by the regularity of the extended problem (\ref{problem:extended}), in the sense of Definition \ref{def: regularity}. More details can be found in \cite{Khalil et al. IEEE TAC 2019}. \end{remark}

We explicit now these necessary conditions (i)-(v).

Condition (i) is equivalent to the following:

\begin{align*}
\dot \psi(t) &  =\left(-\psi(t) \frac{\partial v}{\partial x} (x_0^*(t),x^*(t)) + \mu(t) \frac{\partial \Gamma}{\partial x}(x_0^*(t),x^*(t),u^*(t))\right).u_0^*(t) \\ 
& =-\psi(t) \frac{\partial v}{\partial x} (t,x^*(t)) + \mu(t) \frac{\partial \Gamma}{\partial x}(t,x^*(t),u^*(t))
\end{align*}
which proves condition (a) of Theorem \ref{theorem: maximum principle}, and
\begin{align*}
\dot \psi_0(t) &  =\bigg(-\psi(t) \frac{\partial v}{\partial x_0} (x_0^*(t),x^*(t)) + \mu(t) \frac{\partial \Gamma}{\partial x_0}(x_0^*(t),x^*(t),u^*(t))\bigg).u_0^*(t) \\ 
& =-\psi(t) \frac{\partial v}{\partial x_0} (t,x^*(t)) + \mu(t) \frac{\partial \Gamma}{\partial x_0}(t, x^*(t),u^*(t)).
\end{align*}

Expliciting the maximum condition (ii) implies
\begin{align} \nonumber
&\max\limits_{u_0\in [1-\alpha, 1+\alpha], \ u \in U} \left\{ u_0 \left( \psi_0(t) + (\psi_1(t)-\mu(t)) u_1 + \psi_2(t) u_2  \right)  \right\} \\ 
&  \label{eq: adjoint system extra adjoint variable}\qquad \qquad  = \psi_0(t) + (\psi_1(t)-\mu(t)) u^*_1(t) + \psi_2(t) u_2^*(t),
\end{align}
i.e.
\[  \psi_0(t) + (\psi_1(t)-\mu(t)) u^*_1(t) + \psi_2(t) u_2^*(t) \ge   u_0 \left( \psi_0(t) + (\psi_1(t)-\mu(t)) u_1 + \psi_2(t) u_2  \right) \]
for all $u\in U$ and $u_0 \in [1-\alpha, 1+\alpha]$, in particular for $u_0=1$. Therefore, for all $u\in U$
\[ (\psi_1(t)-\mu(t)) u^*_1(t) + \psi_2(t) u_2^*(t) \ge (\psi_1(t)-\mu(t)) u_1 + \psi_2(t) u_2 \]
which confirms that the maximum condition (b) of Theorem \ref{theorem: maximum principle} holds true.

Furthermore, the maximum condition (ii) is equivalent to
\[ u_0(\psi_0+\bar H(u))\le\psi_0+\bar H(\bar u)\qquad\forall u_0 \in [1-\alpha, 1+\alpha],\ u \in U. \]
This allows us to deduce that $\psi_0 = - \bar H(\bar u)$. Therefore, owing to the adjoint equation (\ref{eq: adjoint system extra adjoint variable}) associated to $\psi_0$, we obtain
\begin{align*}
\dot \psi_0(t) = \diff{\psi_0}{t} (t) = - \diff{\bar H(\bar u)}{t} (t)= \pdv{(-\bar H_0)}{x_0} ({x}_0^*, u_0^*, x^*,u^*,\psi_0,\psi,\mu,\lambda)
\end{align*}
However, $\bar H_0 = \bar H u_0 + \psi_0 u_0$. We deduce that
\[\diff{\bar H(u^*)}{t} = \pdv{\bar H}{t} (u^*) \]
confirming the time-transversality condition (c) of Theorem \ref{theorem: maximum principle}.

The non-triviality condition (e) is a direct consequence of (\ref{eq: result from the non-triviality}). It suffices to prove it by contradiction. Finally, condition (d) is direct from condition (iv). Therefore, Theorem \ref{theorem: maximum principle} is proved.

\vskip3ex
{\bf Proof of Condition (\ref{eq: result from the non-triviality}).} We recall that condition (\ref{eq: result from the non-triviality}) is
\begin{equation*} |\psi_1(t)-\mu(t)| + |\psi_2(t)|  > 0 \quad \text{for all } t \in [0,T^*].  \end{equation*}

 Condition (v) above implies that	
\begin{equation}\label{item: non-trivaility condition for the extended problem}
|\psi_0(t)| + |\psi_1(t)-\mu(t)| + |\psi_2(t)|  > 0, \qquad \text{for all } t\in [0,T^*].
\end{equation}
Indeed, by contradiction, if there exists some $t_1 \in [0,T^*]$ such that (\ref{item: non-trivaility condition for the extended problem}) is violated, then the conservation law (iii)
\[  \max_{u_0\in [1-\alpha, 1+\alpha], \ \varphi(u) \le 0} \{ (\psi_1-\mu)(u_1+v_1)u_0 + \psi_2(u_2+v_2)u_0 + \psi_0 u_0 \} =  \lambda\]
implies that $\lambda=0$, contradicting (v). Now to prove (\ref{eq: result from the non-triviality}), we suppose that there exists $t_1\in [0,T^*]$ such that (2) is violated. Then, replacing in the maximum condition (ii), we obtain
\[  \psi_0(t_1) = \max_{u_0\in [1-\alpha, 1+\alpha]} \psi_0(t_1) u_0  \]
which implies that $\psi_0(t_1)=0$ (recalling that $u_0\neq 0$). Therefore, (\ref{item: non-trivaility condition for the extended problem}) is violated. This proves that (\ref{eq: result from the non-triviality}) holds true.

\end{document}